\documentclass[11pt,a4paper,english]{article}
\usepackage[T1]{fontenc}
\usepackage[latin9]{inputenc}
\usepackage{babel}
\usepackage{amsthm}
\usepackage{amsmath}
\usepackage{amssymb}
\usepackage{esint}
\usepackage[unicode=true,pdfusetitle,
 bookmarks=true,bookmarksnumbered=false,bookmarksopen=false,
 breaklinks=false,pdfborder={0 0 1},backref=false,colorlinks=false]
 {hyperref}

\makeatletter

\pdfpageheight\paperheight
\pdfpagewidth\paperwidth

\numberwithin{equation}{section}
\numberwithin{figure}{section}
\theoremstyle{plain}
\newtheorem{thm}{\protect\theoremname}[section]
  \theoremstyle{definition}
  \newtheorem{defn}[thm]{\protect\definitionname}
  \theoremstyle{plain}
  \newtheorem{lem}[thm]{\protect\lemmaname}
  \theoremstyle{remark}
  \newtheorem*{acknowledgement*}{\protect\acknowledgementname}
  \theoremstyle{remark}
  \newtheorem{rem}[thm]{\protect\remarkname}
  \theoremstyle{plain}
  \newtheorem{prop}[thm]{\protect\propositionname}
  \theoremstyle{plain}
  \newtheorem{cor}[thm]{\protect\corollaryname}

\@ifundefined{date}{}{\date{}}
\usepackage{a4wide}

\usepackage{type1cm}
\usepackage[T1]{fontenc}
\usepackage{courier}
\usepackage{sectsty}
\allsectionsfont{\usefont{OT1}{ptm}{bc}{n}\selectfont}
\usepackage{bbold}
\usepackage{calrsfs}

\renewenvironment{proof}[1][\proofname]{\par
      \pushQED{\qed}%
      \normalfont \topsep6\p@\@plus6\p@\relax
      \trivlist
      \item[\hskip\labelsep
            \bfseries
        #1\@addpunct{.}]\ignorespaces
    }{%
  \popQED\endtrivlist\@endpefalse
}

\makeatother

  \providecommand{\acknowledgementname}{Acknowledgement}
  \providecommand{\corollaryname}{Corollary}
  \providecommand{\definitionname}{Definition}
  \providecommand{\lemmaname}{Lemma}
  \providecommand{\propositionname}{Proposition}
  \providecommand{\remarkname}{Remark}
\providecommand{\theoremname}{Theorem}

\begin{document}

\title{Translated points and Rabinowitz Floer homology}

\author{Peter Albers and Will J. Merry}

%\address{tat}

\maketitle
\begin{abstract}
We prove that if a contact manifold admits an exact filling then every
local contactomorphism isotopic to the identity admits a translated
point in the interior of its support, in the sense of Sandon \cite{Sandon_On_iterated_translated_points_for_contactomorphisms_of_R2n+1_and_R2nxS1}. In addition we prove that if the Rabinowitz Floer homology of the
filling is non-zero then every contactomorphism isotopic to the identity
admits a translated point, and if the Rabinowitz Floer homology of
the filling is infinite dimensional then every contactmorphism isotopic
to the identity has either infinitely many translated points, or a
translated point on a closed leaf. Moreover if the contact manifold
has dimension greater than or equal to $3$, the latter option generically
doesn't happen. Finally, we prove that a generic compactly supported contactomorphism on $\mathbb{R}^{2n+1}$ has infinitely many geometrically distinct iterated translated points all of which lie in the interior of its support.
\end{abstract}

\section{Introduction}

Let $(\Sigma^{2n-1},\xi)$ denote a coorientable closed connected contact manifold,
and let $\alpha$ denote a 1-form on $\Sigma$ such that $\xi=\ker\,\alpha$.
Let $R_{\alpha}$ denote the Reeb vector field of $\alpha$, and let
$\phi_{t}^{\alpha}:\Sigma\rightarrow\Sigma$ denote the flow of $R_{\alpha}$. 

Denote by $\mbox{Cont}(\Sigma,\xi)$ the group of contactomorphisms
$\psi:\Sigma\rightarrow\Sigma$, and denote by $\mbox{Cont}_{0}(\Sigma,\xi)\subseteq\mbox{Cont}(\Sigma,\xi)$
those contactomorphisms $\psi$ that are contact isotopic to $\mathbb{1}$. 
\begin{defn}
Fix $\psi\in\mbox{Cont}(\Sigma,\xi)$, and write $\psi^{*}\alpha=\rho\alpha$
for $\rho\in C^{\infty}(\Sigma,\mathbb{R}^{+})$. We say that a point
$x\in\Sigma$ is a \textbf{translated point }for $\psi$ if
there exists $\tau\in\mathbb{R}$ such that 
\[
\psi(x) = \phi_{\tau}^{\alpha}(x)\ \ \ \mbox{and}\ \ \ \rho(x)=1.
\]

If $x$ is not a periodic point of the Reeb flow $\phi_{t}^{\alpha}$
then such a $\tau$ is uniquely determined, and in this case we call
$\tau$ the \textbf{time shift }of $x$. If $x$ is a periodic point
then $\tau$ is no longer uniquely determined - for if $\phi_{T}^{\alpha}(x)=x$
then $\psi(x)=\phi_{\nu T+\tau}^{\alpha}(x)$ for each $\nu\in\mathbb{N}$.
We say that a point
$x\in\Sigma$ is an \textbf{iterated translated point }for $\psi$ if it is a translated point for some iteration $\psi^n$.
\end{defn}
The notion of (iterated) translated points was introduced by Sandon in \cite{Sandon_On_iterated_translated_points_for_contactomorphisms_of_R2n+1_and_R2nxS1}
and further explored in \cite{Sandon_A_Morse_estimate_for_translated_points_of_contactomorphisms_of_spheres_and_projectives_paces}. We refer to the reader
to these papers for a discussion as to why translated points are a
worthwhile concept to study.

Let $\xi_{\textrm{st}}$ denote the standard contact structure on
$\mathbb{R}^{2n-1}$. Suppose $\sigma:\mathbb{R}^{2n-1}\rightarrow\mathbb{R}^{2n-1}$
is a contactomorphism such that $\mathfrak{S}(\sigma):=\mbox{supp}(\sigma)$
is compact, and suppose that $\mathtt{x}:\mathbb{R}^{2n-1}\rightarrow U\subseteq\Sigma$
is a Darboux chart onto an open subset $U$ of $\Sigma$. Then we
can form a contactomorphism $\psi:\Sigma\rightarrow\Sigma$ such that
$\psi=\mathtt{x}\circ\sigma\circ\mathtt{x}^{-1}$ on $U$ and $\psi=\mathbb{1}$ on $\Sigma\backslash U$.
We call $\psi$ the the \textbf{local contactomorphism }induced from
$\sigma$. In this case we are only interested in translated points
of $\psi$ in the interior of $\mathfrak{S}(\psi)$. Indeed, if $x\in\Sigma\backslash\mathfrak{S}(\psi)$
then $x$ is vacuously a translated point of $\psi$. 

\begin{rem}\label{rem:local_cont}
Since a ball of arbitrary radius in $\mathbb{R}^{2n-1}$ is contactomorphic to $\mathbb{R}^{2n-1}$, see \cite{Chekanov_Koert_Schlenk_Minimal_atlases_of_closed_contact_manifolds}, any contactomorphism of $\mathbb{R}^{2n-1}$ gives rise to a local contactomorphism via an appropriate Darboux chart.
\end{rem}

\begin{defn}
We say that the coorientable closed connected contact manifold $(\Sigma,\xi=\ker\,\alpha)$ admits an \textbf{exact filling
}if there exists a compact symplectic manifold $(M,d\lambda_{M})$
such that $\Sigma:=\partial M$ and such that $\alpha=\lambda_{M}|_{\Sigma}$. 
\end{defn}
In this case let us denote by $X:=M\cup_{\Sigma}\left\{ \Sigma\times[1,\infty)\right\} $
the \textbf{completion} of $M$. Define 
\[
\lambda:=\begin{cases}
\lambda_{M}, & \mbox{on }M,\\
r\alpha, & \mbox{on }\Sigma\times[1,\infty).
\end{cases}
\]
Then $(X,d\lambda)$ is an exact symplectic manifold that is convex at infinity, and $\Sigma\subseteq X$
is a hypersurface of restricted contact type. 

Since $X$ is an exact symplectic manifold that is convex at infinity and $\Sigma$ is a hypersurface
of restricted contact type, the \textbf{Rabinowitz Floer homology
}of the pair $(\Sigma,X)$ is a well defined $\mathbb{Z}_{2}$-vector
space. Rabinowitz Floer homology was discovered by Cieliebak and Frauenfelder
in \cite{Cieliebak_Frauenfelder_Restrictions_to_displaceable_exact_contact_embeddings}, and has since generated many
applications in symplectic topology (we refer to the survey article
\cite{Albers_Frauenfelder_RFH_Survey} for more information on Rabinowitz
Floer homology). 

We can now state our main result. 
\begin{thm}
\label{thm:main}Suppose $(\Sigma,\xi)$ is a closed contact manifold
admitting an exact filling $(M,d\lambda_{M})$. Then:
\begin{enumerate}
\item If $\psi\in\mbox{\emph{Cont}}_{0}(\Sigma,\xi)\setminus\{\mathbb{1}\}$ is a local contactomorphism
then $\psi$ has translated point $x\in\mbox{\emph{int}}(\mathfrak{S}(\psi))$.
\item If the Rabinowitz Floer homology $\mbox{\emph{RFH}}(\Sigma,X)$ does
not vanish then every $\psi\in\mbox{\emph{Cont}}_{0}(\Sigma,\xi)$ has
a translated point. 
\item If $\mbox{\emph{RFH}}(\Sigma,X)$ is infinite dimensional then for
$\psi\in\mbox{\emph{Cont}}_{0}(\Sigma,\xi)$ either $\psi$ has infinitely
many translated points or $\psi$ has a translated point lying on
a closed leaf of $R_{\alpha}$.
\item If $\dim\,\Sigma\geq3$ and $\Sigma$ is non-degenerate then a generic $\psi\in\mbox{\emph{Cont}}_{0}(\Sigma,\xi)$ has no translated point lying on a closed leaf of $R_{\alpha}$.
\item For a generic $\psi\in\mbox{\emph{Cont}}_{0}(\Sigma,\xi)$ the following holds. If $x\in\Sigma$ is a translated point for $\psi^n$, $n\in\mathbb{N}$, then $x$ is \textbf{\emph{not}} a translated point for $\psi,\psi^2,\ldots,\psi^{n-1}$.
\end{enumerate}
\end{thm}

\begin{rem}
Property 5.~in Theorem \ref{thm:main} holds in fact for leafwise intersections as the proof will show. 
\end{rem}

The following corollary is well-known and follows from Chekanov's work \cite{Chekanov_Critical_points_of_quasifunctions_and_generating_families_of_Legendrian_manifolds}.

\begin{cor}
Any $\psi\in\mbox{\emph{Cont}}_{0}(\mathbb{R}^{2n-1},\xi_{\textrm{\emph{st}}})\setminus\{\mathbb{1}\}$ admits a translated point $x\in\mbox{\emph{int}}(\mathfrak{S}(\psi))$.
\end{cor}

\begin{proof}
This follows from Theorem \ref{thm:main} together with Remark \ref{rem:local_cont}.
\end{proof}

\begin{cor}
A generic $\psi\in\mbox{\emph{Cont}}_{0}(\mathbb{R}^{2n-1},\xi_{\textrm{\emph{st}}})\setminus\{\mathbb{1}\}$ admits infinitely many geometrically distinct iterated translated points all of which lie in  $\mbox{\emph{int}}(\mathfrak{S}(\psi))$.
\end{cor}

\begin{proof}
By the previous corollary every $\psi^n$ admits a translated point $x_n\in\mbox{int}(\mathfrak{S}(\psi^n))$. By property 5.~in Theorem \ref{thm:main} the set $\{x_{n}\,:\,n\in\mathbb{N}\}$ cannot be finite for a generic $\psi$.
\end{proof}

\begin{rem}
Sandon proved in \cite{Sandon_A_Morse_estimate_for_translated_points_of_contactomorphisms_of_spheres_and_projectives_paces} that any \textbf{positive} $\psi$ admits infinitely many geometrically distinct iterated translated points.
\end{rem}

In order to explain the idea behind the proof of Theorem \ref{thm:main},
we need to introduce a few more definitions. Recall that from $(\Sigma,\xi)$ we can build the \textbf{symplectization
}of $\Sigma$, which is the exact symplectic manifold $(S\Sigma,d(r\alpha))$,
where
\[
S\Sigma:=\Sigma\times\mathbb{R}^{+},
\]
and $r$ is the coordinate on $\mathbb{R}^{+}:=(0,\infty)$. Suppose
$\psi\in\mbox{Cont}(\Sigma,\xi)$. There exists a unique positive
smooth function $\rho\in C^{\infty}(\Sigma,\mathbb{R}^{+})$ such
that $\psi^{*}\alpha=\rho\alpha$. We define the \textbf{symplectization
}of $\psi$ to be the symplectomorphism $\varphi:S\Sigma\rightarrow S\Sigma$
defined by 
\[\label{eqn:phi}
\varphi(x,r)=(\psi(x),r\rho(x)^{-1}).
\]
Let us now go back to the completion $X$ of $M$. Let $Y_{M}$ denote
the Liouville vector field of $\lambda_{M}$ (defined by $i_{Y_{M}}d\lambda_{M}=\lambda_{M}$).
The entire symplectization $S\Sigma$ embeds into $X$ via the flow
of $Y_{M}$, and under this embedding the vector field $Y$ on $X$
defined by 
\[
Y:=\begin{cases}
Y_{M}, & \mbox{on }M,\\
r\partial_{r}, & \mbox{on }S\Sigma
\end{cases}
\]
satisfies $i_{Y}d\lambda=\lambda$ on all of $X$. Note that under
this embedding $S\Sigma\hookrightarrow X$, the hypersurface $\Sigma\times\{1\}$
in $S\Sigma$ is identified with $\Sigma$ in $X$. 

Suppose $\varphi\in\mbox{Symp}(X,\omega)$. A point $x\in\Sigma$
is called a \textbf{leaf-wise intersection point }for $(\Sigma,\varphi)$
if there exists $\tau\in\mathbb{R}$ such that 
\[
\varphi(x,1)=(\phi_{\tau}^{\alpha}(x),1).
\]
We point out that this definition still makes sense if $\varphi$ is only defined on
$S\Sigma\subseteq X$ rather than on all of $X$. 
%In this case it
%is more convenient to phrase the definition using the hypersurface
%$\Sigma\times\{1\}$ inside $S\Sigma$. Thus if $\varphi\in\mbox{Symp}(S\Sigma,d(r\alpha))$
%then a point $x\in\Sigma$ is called a leaf-wise intersection point
%for $(\Sigma,\varphi)$ if there exists $\tau\in\mathbb{R}$ such
%that 
%\[
%\varphi(x,1)=(\phi_{\tau}^{\alpha}(x),1).
%\]
Our starting point is the following observation of Sandon \cite{Sandon_On_iterated_translated_points_for_contactomorphisms_of_R2n+1_and_R2nxS1}.
\begin{lem}
\label{lem:trans =00003D le}Fix $\psi\in\mbox{\emph{Cont}}(\Sigma,\xi)$
and let $\varphi\in\mbox{\emph{Symp}}(S\Sigma,d(r\alpha))$ denote
the symplectization of $\psi$. Then a point $x\in\Sigma$ is a translated
point for $\psi$ if and only if $(x,1)$ is a leaf-wise intersection
point for $\varphi$.
\end{lem}
This reduces the existence problem for translated points of $\psi$
to the existence problem of leaf-wise intersections for $\varphi$.
Urs Frauenfelder and the first author developed in \cite{Albers_Frauenfelder_Leafwise_intersections_and_RFH}
a variational characterization for the leaf-wise intersection problem using the Rabinowitz action functional, and it is precisely this
characterization that we will exploit.\textbf{ }
\begin{acknowledgement*}
We thank Urs Frauenfelder and Sheila Sandon for several useful
comments, and we thank the anonymous referee for their careful readings of our manuscript.
\end{acknowledgement*}

\section{Proofs}

In this article we use the same sign conventions as \cite{Abbo_Schwarz_Estimates_and_computations_in_RFH}.
The symplectic gradient $X_{H}$ of a Hamiltonian $H\in C^{\infty}(X,\mathbb{R})$
is defined by $i_{X_{H}}d\lambda=-dH$. An almost complex structure
$J$ on $X$ is compatible with $d\lambda$ if $d\lambda(J\cdot,\cdot)$
defines a Riemannian metric on $X$ (warning: this sign convention
is \textbf{not }standard). When doing Rabinowitz Floer homology we
work with negative gradient flow lines.\newline
Fix once and for all a 1-form $\alpha\in\Omega^{1}(\Sigma)$ such
that $\xi=\ker\,\alpha$, and fix an exact symplectic filling $(M,d\lambda_{M})$.
As usual denote the completion of $(M,d\lambda_{M})$ by $(X,d\lambda)$.
We denote by 
\begin{equation}
\wp(\Sigma,\alpha)>0\label{eq:smallest reeb}
\end{equation}
the minimal period of an orbit of $R_{\alpha}$ that is contractible
in $X$.

We denote by $\mathcal{C}(\Sigma,\xi)$ the set of contact isotopies
$\{\psi_{t}\}$ for $t\in\mathbb{R}$ which satisfy $\psi_{0}=\mathbb{1}$
and $\psi_{t+1}=\psi_{t}\circ\psi_{1}$. The universal cover $\widetilde{\mbox{Cont}}_{0}(\Sigma,\xi)$
of $\mbox{Cont}_{0}(\Sigma,\xi)$ consists of equivalence classes
of members of $\mathcal{C}(\Sigma,\xi)$, where two paths $\{\psi_{t}\}$
and $\{\psi_{t}'\}$ are equivalent if there exists a smooth family
$\{\psi_{s,t}\}$ for $(s,t)\in[0,1]\times\mathbb{R}$ such that $\psi_{0,t}=\psi_{t}$
and $\psi_{1,t}=\psi_{t}'$ with $\{\psi_{s,t}\}\in\mathcal{C}(\Sigma,\xi)$
for each $s\in[0,1]$. 

The infinitesimal generator $W$ of a contact isotopy $\{\psi_{t}\}\in\mathcal{C}(\Sigma,\xi)$
is defined by 
\[
W(x):=\frac{\partial}{\partial t}\Bigl|_{t=0}\psi_{t}(x),
\]
and we say that $\{\psi_{t}\}$ is \textbf{generated} by the function
$h:\mathbb{R}/\mathbb{Z}\times\Sigma\rightarrow\mathbb{R}$ defined
by 
\begin{equation}
h(t,x):=\alpha_{\psi_{t}(x)}(W(\psi_{t}(x))\label{eq:the function f}
\end{equation}
($h$ is 1-periodic because $\psi_{t+1}=\psi_{t}\circ\psi_{1}$).

Suppose $\{\psi_{t}\}\in\mathcal{C}(\Sigma,\xi)$. In this case if
$\psi:=\psi_{1}$ then the symplectization $\varphi$ of $\psi$ belongs
to $\mbox{Ham}(S\Sigma,d(r\alpha))$. Indeed, if we define the \textbf{contact
Hamiltonian }$H:\mathbb{R}/\mathbb{Z}\times S\Sigma\rightarrow\mathbb{R}$
associated to $\{\psi_{t}\}$\textbf{ }by 
\[
H(t,x,r):=rh(t,x),
\]
where $h$ is the function from \eqref{eq:the function f}, and we
denote by $\varphi_{t}$ the Hamiltonian flow of $H$, then it can
be shown (see for instance \cite[Proposition 2.3]{Albers_Frauenfelder_A_variational_approach_to_Giventals_nonlinear_Maslov_index})
that 
\[
\varphi_{t}(x,r)=(\psi_{t}(x),r\rho_{t}(x)^{-1}).
\]
Thus the symplectization $\varphi$ of $\psi$ is simply $\varphi_{1}$. We define 
\[
F_{0}:S\Sigma\rightarrow\mathbb{R}
\]
by
\[
F_{0}(x,r):=f(r),
\]
 where 
\[
f(r):=\frac{1}{2}(r^{2}-1)\ \ \ \mbox{on }(\tfrac12,\infty),
\]
\[
f''(r)\geq0\ \ \ \mbox{for all }r\in\mathbb{R}^{+},
\]
\begin{equation}
\lim_{r\rightarrow0}f(r)=-\frac{1}{2}+\varepsilon\label{eq:lim f}
\end{equation}
for some small $\varepsilon>0$. Note that the Hamiltonian vector
field $X_{F_{0}}$ is given by $X_{F_{0}}(x,r)=f'(r)R_{\alpha}(x)$;
in particular $X_{F_{0}}|_{\Sigma\times\{1\}}=R_{\alpha}$.

Let $\chi\in C^{\infty}(S^{1},[0,\infty))$ denote a smooth function such that $\chi>0$ on $(0,\tfrac12)$
and $\chi=0$ on $[\tfrac12,1]$, and such that 
\[
\int_{0}^{1}\chi(t)dt=\int_{0}^{\frac12}\chi(t)dt=1.
\]

Finally fix a smooth function $\vartheta:[0,1]\rightarrow[0,1]$ such
that $\vartheta(t)=0$ for $t\in[0,\tfrac12]$, and such that $\vartheta(1)=1$
with $0\leq\dot{\vartheta}(t)\leq4$ for all $t\in[0,1]$. Denote
by $\mathcal{L}(S\Sigma):=C^{\infty}(S^{1},S\Sigma)$. We now define
the Rabinowitz action functional we will work with. 
\begin{defn}
We define the \textbf{Rabinowitz action functional} 
\[
\mathcal{A}:\mathcal{L}(S\Sigma)\times\mathbb{R}\rightarrow\mathbb{R}
\]
by 
\[
\mathcal{A}(v,\eta):=\int_{0}^{1}v^{*}\lambda-\eta\int_{0}^{1}\chi(t)F_{0}(v)dt-\int_{0}^{1}\dot{\vartheta}(t)H(\vartheta(t),v)dt.
\]

\end{defn}
A simple calculation tells us that if $(v,\eta)\in\mbox{Crit}(\mathcal{A})$
then if we write $v(t)=(x(t),r(t))$ we have 
\[
\begin{cases}
\dot{v}(t)=\eta\chi(t)X_{F_{0}}(v)+\dot{\vartheta}(t)X_{H}(\vartheta(t),v),\\
\int_{0}^{1}\chi(t)F_{0}(v)dt=0.
\end{cases}
\]
The following lemma appears in \cite[Proposition 2.4]{Albers_Frauenfelder_Leafwise_intersections_and_RFH},
and explains the connection between the Rabinowitz action functional
and leaf-wise intersection points (and hence translated points, via
Lemma \ref{lem:trans =00003D le}).
\begin{lem} \label{lem:critical points}Define
\[
e:\mbox{\emph{Crit}}(\mathcal{A})\rightarrow\Sigma
\]
by 
\[
e(v,\eta):=x(\tfrac12)
\]
where $v=(x,r)$. Then $e$ is a surjection onto the set of translated
points for $\psi$. If $\psi$ has no translated points lying on closed
leaves of $R_{\alpha}$ then $e$ is a bijection. If $x(\tfrac12)$ is
not a periodic point of $\phi_{t}^{\alpha}$ then $x(\tfrac12)$ has time
shift $-\eta$. Finally, 
\[
\mathcal{A}(v,\eta)=\eta
\]
for $(v,\eta)\in\mbox{\emph{Crit}}(\mathcal{A})$.\end{lem} \begin{proof}
Let $\varphi:S\Sigma\rightarrow S\Sigma$ denote the symplectization
of $\psi$. Suppose $(v,\eta)\in\mbox{Crit}(\mathcal{A})$. Write
$v(t)=(x(t),r(t))\in\Sigma\times\mathbb{R}^{+}$. For $t\in[0,\tfrac12]$
one has $\dot{v}(t)=\eta\chi(t)X_{F_{0}}(v(t))$. Since $F_{0}$ is
autonomous, $F_{0}$ is constant on flow lines of $\eta\chi X_{F_0}$and thus $t\mapsto F_{0}(v(t))=f(r(t))$
is constant for $t\in[0,\tfrac12]$. The second condition tells this constant
is $0$, and hence $r(t)=1$ for $t\in[0,\tfrac12]$. Thus $v(\tfrac12)=(\phi_{\eta}^{\alpha}(x(0)),1)$.

Next, for $t\in[\tfrac12,1]$ we have $\dot{v}(t)=\dot{\vartheta}(t)X_{H}(\vartheta(t),v(t))$.
Thus $v(t)=\varphi_{\vartheta(t)}(v(\tfrac12))$ for $t\in[\tfrac12,1]$. In
particular, $\varphi(v(\tfrac12))=(\phi_{-\eta}^{\alpha}(v(\tfrac12)),1)$,
and thus $v(\tfrac12)$ is a leaf-wise intersection point of $\varphi$,
and so $x(\tfrac12)$ is a translated point of $\psi$. Moreover if $x(\tfrac12)$
is not a periodic point of $\phi_{t}^{\alpha}$ then it has time shift
$-\eta$. 

Conversely, if $x\in\Sigma$ is a translated point of $\psi$ with
time shift $\tau$ then if $v:S^{1}\rightarrow X$ is defined by 
\[
v(t):=\begin{cases}
\left(\phi_{-\tau\chi(t)}^{\alpha}(\phi_{-\tau}^{\alpha}(x)),1\right), & t\in[0,\tfrac12],\\
\varphi_{\vartheta(t)}(x), & t\in[\tfrac12,1],
\end{cases}
\]
then $v$ is smooth and if $\eta:=-\tau$ then $(v,\eta)\in\mbox{Crit}(\mathcal{A})$.
If there are no translated point lying on closed leaves of $R_{\alpha}$
then these two operations are mutually inverse to each other, and
hence in this case $e$ is a bijection. 

In order to prove the last statement, we first note that 
\[
\lambda(X_{F_{0}}(x,r))=f'(r)\alpha_{x}(R_{\alpha}(x))=f'(r),
\]
\[
\lambda(X_{H}(x,r))=dH(x,r)(r\partial_{r})=H(x,r),
\]
and hence
\begin{align*}
\mathcal{A}(v,\eta) & =\int_{0}^{\frac12}\lambda(\eta\chi X_{F_{0}}(v))dt+\int_{\frac12}^{1}\left[\lambda(\dot{\vartheta}X_{H}(\vartheta(t),v))-\dot{\vartheta}H(\vartheta(t),v)\right]dt\\
 & =\eta+0.
\end{align*}

\end{proof}
Unfortunately, in order to be able to define the Rabinowitz Floer
homology, we cannot work with $\mathcal{A}$ as it is not defined
on all of $\mathcal{L}X\times\mathbb{R}$. In order to rectify this,
we extend $F_{0}$ and $H$ to Hamiltonians defined on all of $X$.
Here are the details. Define 
\[
F:X\rightarrow\mathbb{R}
\]
by setting
\[
F|_{X\backslash S\Sigma}:=-\tfrac12+\varepsilon,
\]
where $\varepsilon>0$ is as in \eqref{eq:lim f}, and defining $F=F_{0}$
on $S\Sigma$. Next, for $c>0$ let $\beta_{c}\in C^{\infty}([0,\infty),[0,1])$
denote a smooth function such that 
\[
\beta_{c}(r)=\begin{cases}
1, & r\in[e^{-c},e^{c}],\\
0, & r\in[0,e^{-2c}]\cup[e^{c}+1,\infty),
\end{cases}
\]
and such that 
\[
0\leq\dot{\beta}_{c}(r)\leq2e^{2c}\ \ \ \mbox{for }r\in[e^{-2c},e^{-c}],
\]
\[
-2\leq\dot{\beta}_{c}(r)\leq0\ \ \ \mbox{for }r\in[e^{c},e^{c}+1].
\]
Then define $H_{c}:[0,1]\times X\rightarrow\mathbb{R}$ by 
\[
H_{c}|_{[0,1]\times(X\backslash S\Sigma)}:=0,
\]
and for $(t,x,r)\in[0,1]\times S\Sigma$, 
\[
H_{c}(t,x,r):=\beta_{c}(r)r\dot{\vartheta}(t)h(\vartheta(t),x).
\]

\begin{rem}\label{rmk:above_which_is_cut_off}
Note that for any $c>0$, $H_{c}$ is a compactly supported 1-periodic
Hamiltonian on $X$ with the property that $H_{c}(t,\cdot,\cdot)=0$
for $t\in[0,\tfrac12]$. Moreover the \textbf{Hofer norm }$\left\Vert H_{c}\right\Vert $
of $H_{c}$ satisfies 
\[
\left\Vert H_{c}\right\Vert \leq4(e^{c}+1)(h_{+}+h_{-}),
\]
where 
\[
h_{+}:=\max_{(t,x)\in\mathbb{R}/\mathbb{Z}\times\Sigma}h(t,x),\ \ \ h_{-}:=-\min_{(t,x)\in\mathbb{R}/\mathbb{Z}\times\Sigma}h(t,x).
\]

\end{rem}

\begin{rem}
\label{loca}Suppose $\sigma:\mathbb{R}^{2n-1}\rightarrow\mathbb{R}^{2n-1}$
is a contactomorphism such that $\mathfrak{S}(\sigma)$ is compact,
and suppose that $\mathtt{x}:\mathbb{R}^{2n-1}\rightarrow U\subseteq\Sigma$
is a Darboux chart onto an open subset $U$ of $\Sigma$. Let $\psi:\Sigma\rightarrow\Sigma$
denote the local contactomorphism such that $\psi=\sigma\circ\mathtt{x}$
on $U$ and $\psi=\mathbb{1}$ on $\Sigma\backslash U$. Given $R>0$,
let $\tau_{R}:\mathbb{R}^{2n-1}\rightarrow\mathbb{R}^{2n-1}$ denote
the contact rescaling defined by $\tau_{R}(\mathbf{x},\mathbf{y},z)=(R\mathbf{x},R\mathbf{y},R^{2}z)$
for $(\mathbf{x},\mathbf{y},z)\in\mathbb{R}^{n-1}\times\mathbb{R}^{n-1}\times\mathbb{R}$.
There is a 1-1 correspondence between the translated points of $\sigma$
and the translated points of the conjugation $\sigma_{R}:=\tau_R\circ\sigma\circ\tau^{-1}_{R}$ as follows: $(\mathbf{x},\mathbf{y},z)$ is a translated point of $\sigma$ if and only if $\tau_R(\mathbf{x},\mathbf{y},z)$ is a translated point of $\sigma_R$. Moreover if $\sigma$ is generated by the function $h(t,x)$ then $\sigma_R$ is generated by $R^2h(t,\tau_R^{-1}(x))$.

We denote by $\psi_{R}$ the local contactomorphism of $\Sigma$
corresponding to $\sigma_{R}$ and the function $H_{c,R}$ corresponding to $\psi_R$. Then for fixed $c>0$ we can choose $R$ so small that the Hofer norm of $H_{c,R}$ is smaller than $\wp(\Sigma,\alpha)$.
\end{rem}
We now extend the Rabinowitz action functional $\mathcal{A}$ to a
new functional
\[
\mathcal{A}_{c}:\mathcal{L}X\times\mathbb{R}\rightarrow\mathbb{R}
\]
by
\[
\mathcal{A}_{c}(v,\eta):=\int_{0}^{1}v^{*}\lambda-\eta\int_{0}^{1}\chi(t)F(v)dt-\int_{0}^{1}H_{c}(t,v)dt.
\]
The following result is the key to the present paper. The proof is
similar to (but simpler than) \cite[Proposition 4.3]{Albers_Frauenfelder_A_variational_approach_to_Giventals_nonlinear_Maslov_index}.
\begin{prop}
\label{prop:linfinity-1}There exists $c_0>0$ such that if $c>c_{0}$
then if $(v,\eta)\in\mbox{\emph{Crit}}(\mathcal{A}_{c})$ then $v(S^{1})\subseteq\Sigma\times\mathbb{R}^{+}$,
and moreover if we write $v(t)=(x(t),r(t))$ then $r(S^{1})\subseteq(e^{-c/2},e^{c/2})$.\end{prop}
\begin{proof}
We know that $r(t)=1\in(e^{-c/2},e^{c/2})$ for all $t\in[0,\tfrac12]$.
Thus if 
\[
I:=\left\{ t\in S^{1}\,:\, r(t)\in(e^{-c/2},c^{c/2})\right\} 
\]
 then $I$ is a non-empty open interval containing the interval $[0,\tfrac12]$.
Let $I_{0}\subseteq I$ denote the connected component containing
0. We show that $I_{0}$ is closed, whence $I_{0}=I=[0,1]$. 

If $v(t)\in\Sigma\times(e^{-c},e^{c})$ and $t\in[\tfrac12,1]$ then $r(t)$
satisfies the equation
\[
\dot{r}(t)=-\dot{\vartheta}(t)\frac{\dot{\rho}_{\vartheta(t)}(x(t))}{\rho_{\vartheta(t)}^2(x(t))}\cdot r(t),
\]
cp~equation \eqref{eqn:phi}. Set 
\[
C:=\max\left\{ \left|\frac{\dot{\rho}_{t}(x)}{\rho_{t}^2(x)}\right|\,:\,(t,x)\in[0,1]\times\Sigma\right\} .
\]
Since $0\leq\dot{\vartheta}\leq4$, we see that for $t\in I_{0}\cap[\tfrac12,1]$
it holds that 
\[
e^{-4C}\leq r(t)\leq e^{4C}.
\]
In particular, provided $c>c_{0}:=8C$ then we have that if $v(t)\in\Sigma\times(e^{-c},e^{c})$
then actually $v(t)\in\Sigma\times(e^{-c/2},e^{c/2})$. This shows
that $I_{0}$ is closed as required. 
\end{proof}
As an immediate corollary, we obtain:
\begin{cor}
\label{cor:For--the}For $c>c_{0}$ the critical point equation and
the critical values for critical points of $\mathcal{A}_{c}$ are
independent of $c$. In fact, they agree with those of $\mathcal{A}$.\end{cor}
\begin{rem}
\label{It-is-important}It is important to note that if $\psi$ is
a local contactomorphism then the constant $c_{0}=c_{0}(\psi)$ is
invariant under the contact rescaling $\tau_{R}$ from Remark \ref{loca}.
More precisely, if $\psi_{R}$ is the rescaling of $\psi$ as in Remark
\ref{loca}, then $c_{0}(\psi)=c_{0}(\psi_{R})$. 
\end{rem}
Fix a family $\mathbf{J}=(J_{t})_{t\in S^{1}}$ of $\omega$-compatible
almost complex structures on $X$ such that the restriction $J_{t}|_{\Sigma\times[1,\infty)}$
is of SFT-type (see \cite{Cieliebak_Frauenfelder_Oancea_Rabinowitz_Floer_homology_and_symplectic_homology}). From
$\mathbf{J}$ we obtain an $L^{2}$-inner product $\left\langle \left\langle \cdot,\cdot\right\rangle \right\rangle _{\mathbf{J}}$
on $\mathcal{L}X\times\mathbb{R}$ by
\[
\left\langle \left\langle (\zeta,l),(\zeta',l')\right\rangle \right\rangle _{\mathbf{J}}:=\int_{0}^{1}\omega(J_{t}\zeta(t),\zeta'(t))dt+ll'.
\]
We denote by $\nabla_{\mathbf{J}}\mathcal{A}_{c}$ the gradient of
$\mathcal{A}_{c}$ with respect to $\left\langle \left\langle \cdot,\cdot\right\rangle \right\rangle _{\mathbf{J}}$.
Given $-\infty<a<b<\infty$ we denote by $\mathcal{M}_{\mathbf{J}}(\mathcal{A}_{c})_{a}^{b}$
the set of smooth maps $u=(v,\eta)\in C^{\infty}(\mathbb{R},\mathcal{L}X)$
that satisfy 
\[
\partial_{s}u+\nabla_{\mathbf{J}}\mathcal{A}_{c}(u(s))=0,
\]
\[
a<\mathcal{A}_{c}(u(s))<b\ \ \ \mbox{for all }s\in\mathbb{R}.
\]

The following result is by now standard (see for instance \cite[Proposition 2.5]{Abbo_Schwarz_Estimates_and_computations_in_RFH}
and \cite[Theorem 2.9]{Albers_Frauenfelder_Leafwise_intersections_and_RFH}).
\begin{prop}
\label{prop:linfinity}Given $-\infty<a<b<\infty$ and $\mathbf{J}$
as above, if $c>c_{0}$ then there exists a compact set $K=K(c,\mathbf{J},a,b)\subseteq X\times\mathbb{R}$
such that for all $u=(v,\eta)\in\mathcal{M}_{\mathbf{J}}(\mathcal{A}_{c})_{a}^{b}$
one has 
\[
u(\mathbb{R}\times S^{1})\subseteq K.
\]

\end{prop}
Theorem \ref{thm:main} follows from Proposition \ref{prop:linfinity}
by arguments from \cite{Albers_Frauenfelder_Leafwise_intersections_and_RFH},
as we now explain.
\begin{proof}[Proof of Theorem \ref{thm:main}]
%\emph{(of Theorem \ref{thm:main})}

(1.) Suppose $\psi$ is a local contactomorphism. Fix $c>c_{0}$.
After possibly replacing $\psi$ by $\psi_{R}$ for some $R$ sufficiently
large (see Remark \ref{loca}) we may assume $\left\Vert H_{c}\right\Vert <\wp(\Sigma,\alpha)$
(where $H_{c}$ is the Hamiltonian corresponding to $\psi_{R}$ -
note we are implicity using Remark \ref{It-is-important} here). 

It follows from the proof of Theorem A in \cite{Albers_Frauenfelder_Leafwise_intersections_and_RFH}
that there exists a critical point $(v,\eta)$ of $\mathcal{A}_{c}$
with $\left|\eta\right|\leq\left\Vert H_{c}\right\Vert $. Thus the
translated point is neccesarily a genuine translated point of $\psi_{R}$,
that is, $v(\tfrac12)\in\mbox{int}(\mathfrak{S}(\psi_{R}))$. Thus $\psi_{R}$,
and hence $\psi$, has a translated point in the interior of its support.\\

(2.) It follows directly from Proposition \ref{prop:linfinity} that
the \textbf{Rabinowitz Floer homology }$\mbox{RFH}(\mathcal{A}_{c})_{a}^{b}$
is well defined for $c>c_{0}$. Here we are using the fact that the Rabinowitz action
functional is generically Morse. This is proved exactly as in \cite[Appendix A]{Albers_Frauenfelder_Leafwise_intersections_and_RFH}.
The only difference is that we are working with a more restrictive
class of Hamiltonian perturbations (i.e. rather than arbitrary Hamiltonians
with time support in $[\tfrac12,1]$, here we work only with contact Hamiltonians
which have been reparametrized to have time support in $[\tfrac12,1]$),
but the proof still goes through. In fact, the only place in the proof given in \cite{Albers_Frauenfelder_Leafwise_intersections_and_RFH}
where the fact that we are working with a more restrictive class of
Hamiltonian perturbations could conceivably cause problems is in deducing
Equation (A.21) from Equation (A.18) on \cite[p95]{Albers_Frauenfelder_Leafwise_intersections_and_RFH}.
Nevertheless, the reader may check that even in our more restricted
setting Equation (A.21) does indeed follow from Equation (A.18).

Moreover $\mbox{RFH}(\mathcal{A}_{c})_{a}^{b}$ is independent of
the choice of $c>c_{0}$. Thus it makes sense to define $\mbox{RFH}(\{\psi_{t}\},\Sigma,X)$
via 
\[
\mbox{RFH}(\{\psi_{t}\},\Sigma,X):=\underset{a\downarrow-\infty}{\underrightarrow{\lim}}\underset{b\uparrow\infty}{\underleftarrow{\lim}}\mbox{RFH}_{*}(\mathcal{A}_{c})_{a}^{b}.
\]
See \cite{Albers_Frauenfelder_Leafwise_intersections_and_RFH} for more information. In fact,
by arguing as in \cite[Theorem 2.16]{Albers_Frauenfelder_Leafwise_intersections_and_RFH}, we
have 
\[
\mbox{RFH}(\{\psi_{t}\},\Sigma,X)\cong\mbox{RFH}(\Sigma,X),
\]
where $\mbox{RFH}(\Sigma,X)$ denotes the Rabinowitz Floer homology
of $(\Sigma,X)$, as defined in \cite{Cieliebak_Frauenfelder_Restrictions_to_displaceable_exact_contact_embeddings}.
The second statement of Theorem \ref{thm:main} now follows from Lemma
\ref{lem:critical points} and Corollary \ref{cor:For--the}, exactly
as in \cite[Theorem C]{Albers_Frauenfelder_Leafwise_intersections_and_RFH}.\\

(3.) The third statement in Theorem \ref{thm:main} follows from the Main Theorem in \cite{Kang_Survival_of_infinitely_many_critical_points_for_the_Rabinowitz_action_functional}.\\

(4.) The fact that generically one doesn't find
translated points on closed Reeb orbits when $\dim\,\Sigma\geq3$
is proved exactly as in \cite[Theorem 3.3]{Albers_Frauenfelder_Leafwise_Intersections_Are_Generically_Morse}
(as in Statement (2) above, the fact that we are working with a more
restrictive class of Hamiltonian perturbations does not cause complications
here). \\

(5.) Finally, the fifth statement is proved by arguing as follows.
Fix $k\in\{2,3,\dots,\infty\}$. Denote by $\mathcal{H}^{k}$ the
class of $C^{k}$ contact Hamiltonians $H$, reparametrized so that
$H(t,\cdot)=0$ for $t\in[0,\tfrac12]$, which additionally have
been cutoff outside of a neighborhood of $\Sigma\times\{1\}$ as described above Remark \ref{rmk:above_which_is_cut_off}.

%Recall given (any) two Hamiltonians $K_{1},K_{2}$, the composition
%$K_{1}\#K_{2}$ is defined by 
%\[
%(K_{1}\#K_{2})(t,p):=K_{1}(t,p)+K_{2}(t,(\phi_{t}^{K_{1}})^{-1}(p)).
%\]
Given a Hamiltonian $H$ and $m\in\mathbb{Z}$ we set
$$
H^{\#m}(t,x):=mH(mt,x).
$$
If $\varphi$ is the time-1-map of the Hamiltonian flow of $H$ then the time-1-map of the Hamiltonian flow of $H^{\#m}$ is $\varphi^m$. Note that if $H\in\mathcal{H}^{k}$
then $H^{\#m}\in\mathcal{H}^{k}$ for all $m\in\mathbb{N}$. We point out that the assignment $H\mapsto H^{\#m}$ is linear.

Given $H\in\mathcal{H}^{k}$, we denote by $\mathcal{A}_{H}$ the
Rabinowitz action functional 
\[
\mathcal{A}_{H}(v,\eta)=\int_{0}^{1}v^{*}\lambda-\eta\int_{0}^{1}\chi(t)F(v)dt-\int_{0}^{1}H(t,v)dt
\]
(so that the functional $\mathcal{A}_{c}$ would now be written as
$\mathcal{A}_{H_{c}}$). Let $\mathcal{L}=W^{1,2}(S^{1},X)$ and let $\mathcal{E}$ denote
the Banach bundle over $\mathcal{L}$ with fibre $\mathcal{E}_{v}:=L^{2}(S^{1},v^{*}TX)$. Fix $l,m\in\mathbb{N}$. We now define a section 
\[
\sigma:\mathcal{L}\times\mathbb{R}\times\mathcal{L}\times\mathbb{R}\times\mathcal{H}^{k}\rightarrow\mathcal{E}^{\vee}\times\mathbb{R}\times\mathcal{E}^{\vee}\times\mathbb{R}
\]
by 
\[
\sigma(v,\eta,w,\tau,H):=\left(\mbox{d}\mathcal{A}_{H^{\#l}}(v,\eta),\mbox{d}\mathcal{A}_{H^{\#m}}(w,\tau)\right).
\]
Let $\mathcal{M}:=\sigma^{-1}(\mbox{zero section})$, so that 
\[
\mathcal{M}=\left\{ (v,\eta,w,\tau,H)\,:\,((v,\eta),(w,\tau))\in\mbox{Crit}(\mathcal{A}_{H^{\#l}})\times\mbox{Crit}(\mathcal{A}_{H^{\#m}})\right\} ,
\]
and set 
\[
\mathcal{M}_{H}:=\{(v,\eta,w,\tau)\,:\,(v,\eta,w,\tau,H)\in\mathcal{M}\}=\mbox{Crit}(\mathcal{A}_{H^{\#l}})\times\mbox{Crit}(\mathcal{A}_{H^{\#m}}).
\]
Next, define 
\[
\mathcal{B}:=\left\{ (v,\eta,w,\tau,H)\,:\, v(t)\ne w(t)\mbox{ for all }t\in[\tfrac12,1]\right\} ,
\]
and set 
\[
\mathcal{M}^{*}:=\mathcal{B}\cap\mathcal{M},
\]
\[
\mathcal{M}_{H}^{*}:=\{(v,\eta,w,\tau)\,:\,(v,\eta,w,\tau,H)\in\mathcal{M}^{*}\}.
\]
Consider the evaluation map
\[
\phi_{\textrm{eval}}:\mathcal{L}\times\mathbb{R}\times\mathcal{L}\times\mathbb{R}\times\mathcal{H}^{k}\rightarrow\Sigma\times\Sigma
\]
defined by 
\[
\phi_{\textrm{eval}}(v,\eta,w,\tau,H):=(v(0),w(\tfrac12)).
\]
We would like to prove that there exists a generic set $\mathcal{H}_{l,m}^{k}\subseteq\mathcal{H}^{k}$
with the property that 
\begin{equation}
H\in\mathcal{H}_{l,m}^{k}\ \ \ \Rightarrow\ \ \ \mathcal{A}_{H^{\#l}}\mbox{ and }\mathcal{A}_{H^{\#m}}\mbox{ are Morse and }\phi_{\textrm{eval}}(\cdot,\cdot,\cdot,\cdot,H)|_{\mathcal{M}_{H}^{*}}\ \mbox{is a submersion.}\label{eq:submersion}
\end{equation}
Note that we already know (cf.~(2)) above) that $\mathcal{A}_{H^{\#l}}$
and $\mathcal{A}_{H^{\#m}}$ are Morse for generic $H\in\mathcal{H}^{k}$. 

This is helpful for the following reason: \eqref{eq:submersion} implies
that for $H\in\mathcal{H}_{l,m}^{k}$, if $(v,\eta)$ is a critical
point of $\mathcal{A}_{H^{\#l}}$, and $(w,\tau)$ is a critical point
of $\mathcal{A}_{H^{\#m}}$ which satisfy $v(t)\ne w(t)$ for all
$t\in[\tfrac12,1]$, then $v(0)\ne w(\tfrac12)$. To see this one notes that
if $\phi_{\textrm{eval}}(\cdot,\cdot,\cdot,\cdot,H)|_{\mathcal{M}_{H}^{*}}$
is a submersion then it is certainly transverse to the diagonal $\Delta\subseteq\Sigma\times\Sigma$,
and thus the preimage 
\[
\left(\phi_{\textrm{eval}}(\cdot,\cdot,\cdot,\cdot,H)|_{\mathcal{M}_{H}^{*}}\right)^{-1}(\Delta)\subseteq\mathcal{M}_{H}^{*},
\]
if non-empty, should have codimension $2n-1$. Since $\mathcal{A}_{H^{\#l}}$
and $\mathcal{A}_{H^{\#m}}$ are both Morse, $\mathcal{M}_{H}^{*}$
is zero-dimensional, and thus $\left(\phi_{\textrm{eval}}(\cdot,\cdot,\cdot,\cdot,H)|_{\mathcal{M}_{H}^{*}}\right)^{-1}(\Delta)=\emptyset$,
as required. 

Using this, we complete the proof by making use of the following claim.\newline

\noindent \textbf{Claim:} Set $\mathcal{H}^{*}:=\bigcap_{k,l,m\geq2}\mathcal{H}_{l,m}^{k}$.
Suppose $H\in\mathcal{H}^{*}$, and set $\varphi:=\phi_{1}^{H}$.
Then for all pairs $(l,m)$ of positive integers with $l\ne m$, $\varphi^{l}$
and $\varphi^{m}$ do not have any common leaf-wise intersection points.\\

To prove the claim we argue by contradiction. Without loss of generality
assume $l<m$, and suppose $x\in\Sigma$ is a common leaf-wise intersection
point of $\varphi^{l}$ and $\varphi^{m}$. Thus there exists $\eta,\tau\in\mathbb{R}$
such that 
\[
\varphi^{l}(x)=\phi_{\eta}^{\alpha}(x),\ \ \ \varphi^{m}(x)=\phi_{\tau}^{\alpha}(x).
\]
Then 
\[
\varphi^{m-l}(\varphi^{l}(x))=\phi_{\tau-\eta}^{\alpha}(\varphi^{l}(x)),
\]
so $\varphi^{l}(x)$ is a leaf-wise intersection point of $\varphi^{m-l}$.
Let $(v,-\eta)\in\mbox{Crit}(\mathcal{A}_{H^{\#l}})$ and $(w,-\tau+\eta)\in\mbox{Crit}(\mathcal{A}_{H^{\#m-l}})$
denote the critical points of $\mathcal{A}_{H^{\#l}}$ and $\mathcal{A}_{H^{\#m-l}}$
corresponding to $x$ and $\varphi^{l}(x)$ respectively, so that
$v(0)=\varphi^{l}(x)$ and $v(\tfrac12)=x$, and $w(0)=\varphi^{m}(x)$
and $w(\tfrac12)=\varphi^{l}(x)=v(0)$. By construction $v(t)\ne w(t)$
for all $t\in[\tfrac12,1]$, and this gives the desired contradiction.%\newline

It remains therefore to prove \eqref{eq:submersion}. To prove this
we examine the vertical derivative of $\sigma$, 
\[
D\sigma(v,\eta,w,\tau,H):T_{v}\mathcal{L}\times\mathbb{R}\times T_{w}\mathcal{L}\times\mathbb{R}\times\mathcal{H}^{k}\rightarrow\mathcal{E}_{v}^{\vee}\times\mathbb{R}\times\mathcal{E}_{w}^{\vee}\times\mathbb{R},
\]
which is given by 

$$
\begin{aligned}
(\hat{v},\hat{\eta},\hat{w},\hat{\tau},\hat{H})\mapsto&\left(\mathbf{H}_{\mathcal{A}_{H^{\#l}}}(v,\eta)((\hat{v},\hat{\eta}),\bullet),\mathbf{H}_{\mathcal{A}_{H^{\#m}}}(w,\tau)((\hat{w},\hat{\tau}),\bullet)\right)\\
&+\left(0,\int_{0}^{1}\hat{H}^{\#l}(t,v)dt,0,\int_{0}^{1}\hat{H}^{\#m}(t,w)dt\right), 
\end{aligned}
$$
where, e.g. $\mathbf{H}_{\mathcal{A}_{H^{\#l}}}(v,\eta)$ denotes
the Hessian of $\mathcal{A}_{H^{\#l}}$ at the critical point $(v,\eta)$.
Let 
\[
\mathcal{V}(v,\eta,w,\tau,H)\subseteq T_{v}\mathcal{L}\times\mathbb{R}\times T_{w}\mathcal{L}\times\mathbb{R}\times\mathcal{H}^{k}
\]
denote the subspace of quintuples $(\hat{v},\hat{\eta},\hat{w},\hat{\tau},\hat{H})$
satisfying 
\[
\hat{v}(0)=\hat{w}(\tfrac12)=0.
\]
In order to prove that $\phi_{\textrm{eval}}(\cdot,\cdot,\cdot,\cdot,H)|_{\mathcal{M}_{H}^{*}}$
is a submersion, it suffices to show that the restriction $D\sigma(v,\eta,w,\tau,H)|_{\mathcal{V}(v,\eta,w,\tau,H)}$
is surjective for all $(v,\eta,w,\tau,H)\in\mathcal{M}^{*}$ (see
for instance, \cite[Lemma 3.5]{Albers_Frauenfelder_Leafwise_Intersections_Are_Generically_Morse}).
This can be can be proved exactly as in \cite[Proposition A.2]{Albers_Frauenfelder_Leafwise_intersections_and_RFH},
but for the convenience of the reader we sketch the basic idea here,
since the proof is more complicated than these arguments often are.
In order to show that $D\sigma(v,\eta,w,\tau,H)|_{\mathcal{V}(v,\eta,w,\tau,H)}$
is surjective, it suffices to show that if the annihilator of its image
vanishes. 

The first thing to notice is that the restriction to $\mathcal{V}(v,\eta,w,\tau,H)$
makes no difference here, since as the Hessian of the Rabinowitz action functional (and hence $D\sigma$) is a local operator, the annihilator
of the image of $D\sigma(v,\eta,w,\tau)|_{\mathcal{V}(v,\eta,w,\tau,H)}$
is the same as the annihilator of all of $\mbox{im}\, D\sigma(v,\eta,w,\tau,H)$. More precisely, an element $(\hat{v},\hat{\eta},\hat{w},\hat{\tau})$ in the annihilator of $D\sigma(v,\eta,w,\tau)$ satisfies a certain ODE of the form $P(\hat{v}(t),\hat{\eta},\hat{w}(t),\hat{\tau})=0$ for all $t \in S^1$. See \eqref{eq:P is zero} 
below for one way of formulating the equation $P=0$. An element in the annihilator of  $D\sigma(v,\eta,w,\tau)|_{\mathcal{V}(v,\eta,w,\tau,H)}$ satisfies the same ODE, except possibly at $t=0$ and $t=\tfrac12$. But elliptic boostrapping implies that the solutions of the ODE $P=0$ are of class $C^{k+1}$, and thus by continuity if $(\hat{v},\hat{\eta},\hat{w},\hat{\tau})$ solves the equation $P(\hat{v}(t),\hat{\eta},\hat{w}(t),\hat{\tau})=0$ for all $t \in S^1 \backslash \{0,\tfrac12\}$ then in fact $P(\hat{v}(t),\hat{\eta},\hat{w}(t),\hat{\tau})=0$ also for $t=0,\tfrac12$.

Unfortunately the ODE $P(\hat{v},\hat{\eta},\hat{w},\hat{\tau})=0$ is somewhat messy to write down (and hence solve) directly. The difficulty stems from the fact that we require our Hamiltonians
to have time support in $[\tfrac12,1]$. As explained in \cite[Appendix A]{Albers_Frauenfelder_Leafwise_intersections_and_RFH},
this difficulty can be overcome by translating the problem into one
on \textbf{twisted loop space}s. Here is the definition. Given $H\in \mathcal{H}^k$, set 
\[
\mathcal{L}_{H}:=\left\{ v\in W^{1,2}([0,1],X)\,:\, v(0)=\phi_{1}^{\chi F+H}(v(1))\right\} ,
\]
where $\phi_{t}^{\chi F+H}$ is the Hamiltonian flow of $\chi F+H$. There
is a diffeomorphism $\Phi_{H}:\mathcal{L}_{H}\rightarrow\mathcal{L}$
given by 
\[
\Phi_{H}(v)(t):=\phi_{t}^{\chi F+H}(v(t)).
\]
Define 
\[
\widetilde{\mathcal{A}}_{H}:=\mathcal{A}_{H}\circ(\Phi_{H},\mathbb{1}_{\mathbb{R}}):\mathcal{L}_{H}\times\mathbb{R}\rightarrow\mathbb{R}.
\]
Note that if $(v,\eta)\in\mbox{Crit}(\mathcal{A}_{H})$ then $(v(0),\eta)\in\mbox{Crit}(\widetilde{\mathcal{A}}_{H}),$
where $v(0)$ is thought of as a constant curve in $\mathcal{L}_{H}$.
Let $\mathcal{E}_{H}\rightarrow\mathcal{L}_{H}$ denote the Banach
bundle whose fibre $(\mathcal{E}_{H})_{v}$ is again given by $L^{2}([0,1],v^{*}TX)$. 

Instead of looking at $D\sigma(v,\eta,w,\tau,H)$, we study the transformed
operator
\[
L:T_{v(0)}\mathcal{L}_{H^{\#l}}\times\mathbb{R}\times T_{w(0)}\mathcal{L}_{H^{\#m}}\times\mathbb{R}\times\mathcal{H}^{k}\rightarrow(\mathcal{E}_{H^{\#l}}\times\mathbb{R})^{\vee}\times(\mathcal{E}_{H^{\#m}}\times\mathbb{R})^{\vee}
\]
which is defined by requiring that for tuples $(\hat{v},\hat{\eta},\hat{w},\hat{\tau}),(\hat{v}',\hat{\eta}',\hat{w}',\hat{\tau}')$  in $\mathcal{E}_{H^{\#l}}\times\mathbb{R}\times\mathcal{E}_{H^{\#m}}\times\mathbb{R}$ and $\hat{H}\in \mathcal{H}^k$
the value of 
\[
(L_{1},L_{2}):=\left\langle L(\hat{v}',\hat{\eta}',\hat{w}',\hat{\tau}',\hat{H}),(\hat{v},\hat{\eta},\hat{w},\hat{\tau})\right\rangle \in\mathbb{R}\times\mathbb{R},
\]
is given by
\[
L_{1}=\mathbf{H}_{\widetilde{\mathcal{A}}_{H^{\#l}}}(v(0),\eta)((\hat{v}',\hat{\eta}'),(\hat{v},\hat{\eta}))+\int_{0}^{1}\mbox{d}(\Phi_{H^{\#l}}^{*}\hat{H}^{\#l})(\hat{v}(t))dt,
\]
\[
L_{2}=\mathbf{H}_{\widetilde{\mathcal{A}}_{H^{\#m}}}(w(0),\tau)((\hat{w}',\hat{\tau}'),(\hat{w},\hat{\tau}))+\int_{0}^{1}\mbox{d}(\Phi_{H^{\#m}}^{*}\hat{H}^{\#m})(\hat{w}(t))dt.
\]
Here as before $\mathbf{H}_{\widetilde{\mathcal{A}}_{H^{\#l}}}(v(0),\eta)$
denotes the Hessian of $\widetilde{\mathcal{A}}_{H^{\#l}}$ at the
critical point $(v(0),\eta)$. 

After making this transformation it is not difficult to write down
the equations that  $(\hat{v},\hat{\eta},\hat{w},\hat{\tau})$ must satisfy in order to lie in the annihilator of the image of $L$.  Indeed, firstly by taking $\hat{v}'=\hat{\eta}'=0$ in the equation $L_1=0$ we see that one must have $\hat{v}(t)=0$ for $t\in [0,\tfrac12]$. Secondly, taking $\hat{H}=0$ we see that $(\hat{v},\hat{\eta})$ must satisfy the ODE
\begin{equation}
\partial_{t}\hat{v}-\hat{\eta}\chi(t)X_{F}(v(0))=0,\ \ \ \hat{v}(0)=d\phi_{1}^{\chi F+H^{\#l}}(v(0))(\hat{v}(1)).\label{eq:P is zero}
\end{equation}
Since $\hat{v}(1)=0$, the second equation in \eqref{eq:P is zero} implies that $\hat{v}(0)=0$. Thus integrating the first equation in \eqref{eq:P is zero} tells us that 
\begin{equation}
\hat{v}(t)=\hat{\eta}X_F(v(0))\cdot \int_0^t \chi(s)ds.
\label{eq:P is zero 2}
\end{equation}
Our assumption on $\chi$ thus implies for $t\ge \tfrac12$ one has $0=\hat{v}(t)=\hat{\eta}X_F(v(0))$.  Since $X_F(v(0))\ne0$ we must have $\hat{\eta}=0$, and thus \eqref{eq:P is zero 2} implies that $\hat{v}(t)=0$ for all $t$. 
A similar argument applies to $(\hat{w},\hat{\tau})$, and this finally completes the proof.\end{proof}

\bibliographystyle{amsalpha}
\bibliography{peterbibtex}

\begin{tabular}{p{85ex}}
$ $\\
\noindent\small \textsc{Peter Albers, Mathematisches Institut, WWU M\"unster, Germany}\\[.5ex]

\noindent\small \emph{Email:} \texttt{peter.albers@uni-muenster.de}\\[3ex]

\noindent\small \textsc{Will J.~Merry, Departement Mathematik, ETH Z\"urich, Switzerland}\\[.5ex] 

\noindent\small \emph{Email:} \texttt{merry@math.ethz.ch}

\end{tabular}

\end{document}